\documentclass{birkjour}
 \newtheorem{thm}{Theorem}[section]
 
 \newtheorem{lem}[thm]{Lemma}
 \newtheorem{prop}[thm]{Proposition}
 \theoremstyle{definition}
 
 \theoremstyle{remark}
 \newtheorem{rem}[thm]{Remark}
 
 \numberwithin{equation}{section}

 \usepackage{cite}
\usepackage{hyperref}
\hypersetup{
	colorlinks,
	linkcolor = {blue},
	citecolor = {blue},
	filecolor = {blue},
	urlcolor = {blue} 
} 
\usepackage{amsmath}
\usepackage{url}
\usepackage{float}
\usepackage{amssymb}

\newcommand{\arcangle}{%
  \mathord{<\mspace{-9mu}\mathrel{)}\mspace{2mu}}%
}

\begin{document}

%-------------------------------------------------------------------------
% editorial commands: to be inserted by the editorial office
%
%\firstpage{1} \volume{228} \Copyrightyear{2004} \DOI{003-0001}
%\seriesextraline{This is the Concrete Title of this Book\br H.E. R and S.T.C. W, Eds.}
%\firstpage{1}

\title[Proof of the Generalization of the Sawayama--Thébault Theorem]
 {Proof of the Generalization of the Sawayama--Thébault Theorem}

%----------Author 1
\author[Miłosz Płatek]{Miłosz Płatek}

\address{%
Independent researcher,\\
Krakow, Poland.}

\email{milosz@platek.org}

%----------Author 2



\date{January 1, 2004}

\begin{abstract}
We prove two conjectures posed in 2016 concerning a generalization of the Sawayama--Thébault Theorem and the Sawayama Lemma. We show that this generalized statement can be viewed in Laguerre geometry, which provides a natural framework for resolving the problem.
\end{abstract}

%%% ----------------------------------------------------------------------
\maketitle
%%% ----------------------------------------------------------------------
%\tableofcontents
\section{Introduction}
Among the classical problems of plane geometry, the Sawayama--Thébault Theorem holds a special place due to its long and remarkable history. The problem was posed by the French mathematician Victor Thébault in 1938 as Problem 3887 in The American Mathematical Monthly~\cite{website}.

In 1963, C. Stanley Ogilvy included the problem in his book \textit{Tomorrow’s Math: Unsolved Problems for the Amateur}~\cite{chou}, classifying it as an open problem in analytic geometry.

\begin{prop}[The Sawayama--Thébault Theorem]
Let $ABC$ be a triangle, and let $D$ be a point on the side $BC$ of triangle $ABC$. Let $O_1$ and $O_2$ ($O_1 \neq O_2$) be the centers of circles tangent to segment $BC$, to segment $AD$, and to the circumcircle of triangle $ABC$. Then, regardless of the choice of point $D$, the line $O_1O_2$ passes through a fixed point $I$, which is the incenter of triangle $ABC$.
\end{prop}

\begin{figure}[H]
    \centering
    \includegraphics[scale=1.2]{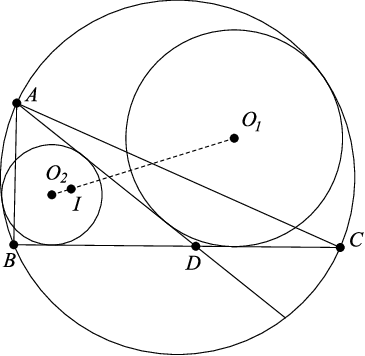}
    \caption{The Sawayama--Thébault Theorem.}
    \label{fig2}
\end{figure}

It was long believed that the first solution was found by K. B. Taylor and published in 1983. His proof consisted of 24 pages of calculations, though only a summary appeared in \textit{The American Mathematical Monthly}~\cite{generalization}. This was the first known solution published in English and, for many years, was widely cited as the earliest proof. Nevertheless, the problem had already been solved in 1973 by Dutch mathematician H. Streefkerk. However, this did not end the search for a more synthetic and elegant solution.

In 1988, Chou~\cite{springer} applied automated theorem proving methods, based on Gröbner bases and pseudo-remainders, to prove the theorem. The original proof required 44 hours of CPU time on a Symbolic 3600 machine. The Sawayama--Thébault Theorem nearly acquired benchmark status in Gröbner basis theory~\cite{t1}.

In 1991, H. Demir and C. Dezar published another version of the proof in \textit{Geometriae Dedicata}, offering a natural solution to the problem and relating it to the classical notable configurations of the triangle~\cite{t6}.

In 2001, R. Shaila published an elementary proof using Cartesian coordinates in \textit{The American Mathematical Monthly}~\cite{t1}. Although the argument was elementary, it still required the use of Maple to compute a root of the quartic equation. As we have noted, there was considerable misconception surrounding the first proof of this theorem. When Shaila published his paper, he stated that he was not aware of any published elementary proof of Thébault’s Theorem.

Finally, in 2003, after 65 years, a short and synthetic proof was published by Jean-Louis Ayme~\cite{circular}. Ayme discovered that Yuzaburo Sawayama, an instructor at the Central Military School in Tokyo, had independently formulated and solved the problem as early as 1905 in his academic work \cite{historia}. Since then, the theorem has been known as the Sawayama--Thébault Theorem, and the lemma used in Sawayama’s original proof is referred to as the Sawayama Lemma.

\begin{prop}\label{p1}
Let $ABC$ be a triangle, and let $D$ be a point on the side $BC$ of triangle $ABC$. Let $\omega$ be a circle tangent to segment $BC$ at point $E$, to segment $AD$ at point $F$, and to the circumcircle of triangle $ABC$. Then, regardless of the choice of point $D$, the line $EF$ passes through a fixed point $I$, which is the incenter of triangle $ABC$.
\end{prop}

\begin{figure}[H]
    \centering
    \includegraphics{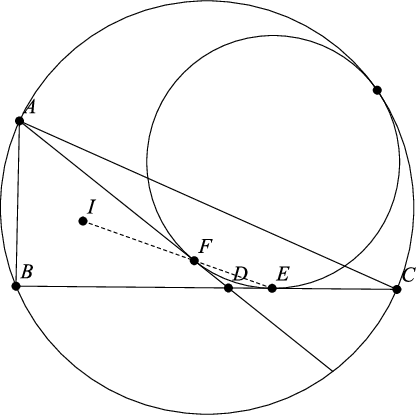}
    \caption{The Sawayama Lemma.}
    \label{fig1}
\end{figure}

Over the years, many of papers on the Sawayama--Thébault Theorem have been published in various mathematical journals \cite{t8,lemma,problem,t9,t2,t5,t10,proof,t3}. The proof has gradually been simplified to a one-page synthetic version based on the Sawayama Lemma, attributed to Y. Sawayama.

Vietnamese mathematician Dao Thanh Oai worked for four years on a generalization of the Sawayama--Thébault Theorem. Many attempts had previously been made to generalize the Sawayama--Thébault Theorem and the Sawayama Lemma. In particular, before discovering his stronger result, Dao proposed another idea for extending the Sawayama Lemma and the Sawayama--Thébault Theorem~\cite{olddao}. One of these results --- the generalization of the Sawayama Lemma --- was later proved by Nguyen Chuong Chi~\cite{nguyen}. However, these generalizations were complicated and fell short of what could be expected from a true generalization. Moreover, in 2002, a version of the theorem involving external circles was proved by S. Gueron and published in \textit{The American Mathematical Monthly}\cite{external}, but this also cannot be regarded as a genuine generalization. After working for four years, Dao finally discovered what could be described as a genuine generalization of the Sawayama--Thébault Theorem and the Sawayama Lemma. This conjecture has remained open until now, for nearly ten years, as confirmed by Dao Thanh Oai.

In this paper, we present proofs of the generalization of the Sawayama Lemma and the generalization of the Sawayama--Thébault Theorem.

A key idea in obtaining the presented proof is to view the problem in terms of Laguerre geometry \cite{laguerre}. The specific Laguerre transformation used is a dilation, which is a special case of an axial circular transformation. This technique is described in \textit{Geometric Transformations IV: Circular Transformations}~\cite{t4}.

In summary, this paper presents proofs of two conjectures: the generalization of the Sawayama--Thébault Theorem and the generalization of the Sawayama Lemma~\cite{t7}. These proofs are synthetic, and their degenerate cases correspond to the classical versions of both theorems, thereby generalizing the results established in the cited papers.

\section{Generalization of the Sawayama Lemma}

In this section, we present a proof of the generalization of the Sawayama Lemma. The conjecture \cite{t7} was proposed by Dao Thanh Oai from Vietnam, who worked on it for four years. We will use the following lemma, which is a variation of the Sawayama lemma. To the author's surprise, it has not been presented in any known source, although the proof can be conducted analogously to the classical Lemma.

\begin{lem}\label{lemma1}
(Sawayama lemma in an alternative configuration)  
Let triangle $ABC$ be given and let point $D$ lie on segment $BC$. Let $\omega$ be a circle tangent to the circumcircle at point $P$, to line $BC$ at point $E$, and to line $AD$ at point $F$, such that points $A$ and $P$ lie on opposite sides of line $BC$, and points $B$ and $P$ lie on the same side of line $AD$. Then, regardless of the choice of point $D$, the line $EF$ passes through point $J_c$, the center of the $C$-excircle of triangle $ABC$.
\end{lem}

\begin{figure}[H]
    \centering
    \includegraphics{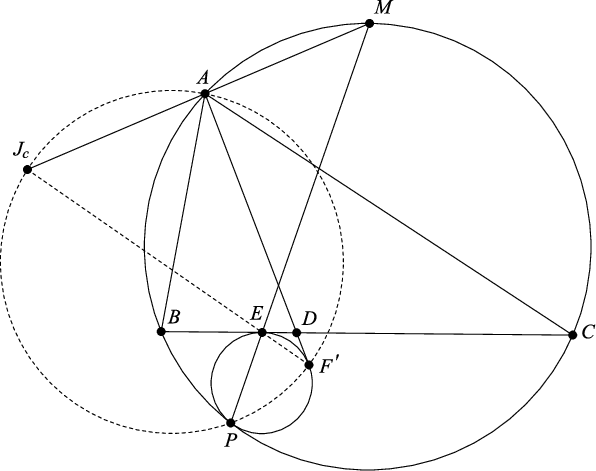}
    \caption{The Sawayama Lemma in an alternative configuration.}
    \label{fig3}
\end{figure}

   \begin{proof}
Let point $F' \neq E$ be the intersection of line $J_cE$ with the circle $\omega$. Let point $M$ be the midpoint of arc $BAC$. Then points $P$, $E$, and $M$ are collinear, since the homothety centered at $P$ that sends the circle $\omega$ to the circumcircle of $\triangle ABC$ maps $E$ to $M$. Again, from the homothety centered at point $P$, we obtain:
\begin{equation}
    \arcangle PF'E = \arcangle PCM = \arcangle J_cAP.
\end{equation}
Thus,
\begin{equation}
    \arcangle J_cAP = \arcangle J_cF'P,
\end{equation}
so quadrilateral $PF'AJ_c$ is cyclic. Moreover, we have
\begin{equation}\label{eq:angles}
    \arcangle MBC = \arcangle BCM = \arcangle MPB.
\end{equation}
Therefore, $MB$ is tangent to the circumcircle of $\triangle BEP$ at $B$. 
By the Power of a Point theorem (see \cite{chen}), we obtain
\begin{equation}\label{eq:power}
    MB^{2} = ME \cdot MP.
\end{equation}
Furthermore, by a variant of the Incenter–Excenter Lemma (see \cite{chen}), we have $MJ_c=MB$. Hence,
\begin{equation}
    ME \cdot MP = MB^2 = MJ_c^2,
\end{equation}
so line $MJ_c$ is tangent to the circumcircle of triangle $J_cEP$.
        
\noindent Therefore,
\begin{equation}
    \arcangle MJ_cE = \arcangle J_cPM = \arcangle ACM + \arcangle J_cF'A.
\end{equation}
However, we also have:
\begin{equation}
    \arcangle MJ_cE = \arcangle APF' = \arcangle ACM + \arcangle EPF'.
\end{equation}
Combining these, we get:
\begin{equation}
    \arcangle ACM + \arcangle EPF' = \arcangle ACM + \arcangle J_cF'A,
\end{equation}
\begin{equation}
    \arcangle  EPF'= \arcangle J_cF'A.
\end{equation}
Hence, line $AF'$ is tangent to circle $\omega$, so $F = F'$, which completes the proof.
\end{proof} 

We will prove one more lemma that will be useful in the proof of the generalization of the Sawayama Lemma.

\begin{lem}\label{lemma2}
Let $ABC$ be a triangle, and let $\omega_a$ be a circle tangent to lines $AC$ and $AB$. Let $\Omega$ be a circle passing through points $B$ and $C$.  Moreover,
  \begin{enumerate}
      \item if the circle $\omega_a$ lies on the opposite side of line $AC$ from point $B$, then the circle $\Omega$ is internally tangent to $\omega_a$,
      \item if the circle $\omega_a$ lies on the same side of line $AC$ as point $B$, then the circle $\Omega$ is externally tangent to $\omega_a$.
  \end{enumerate}
Let $A'$ be the point of tangency of the circles $\omega_a$ and $\Omega$, and let $I$ be the incenter of $\triangle ABC$. Then the line $A'I$ is the angle bisector of $\angle BA'C$.
\end{lem}

\begin{proof}
To simplify the proof, we adopt the first configuration. The proof for the second configuration is analogous.

\begin{figure}
    \centering
    \includegraphics[scale=1.1]{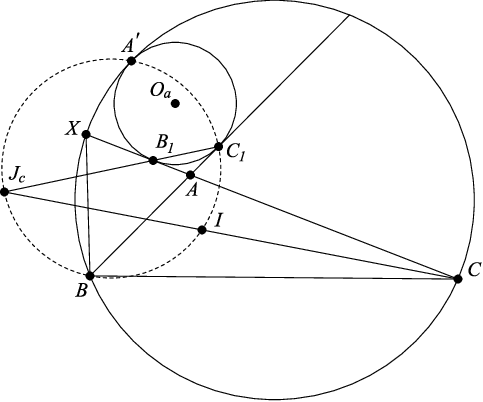}
    \caption{$A'I$ is the angle bisector of angle $BA'C$.}
    \label{fig4}
\end{figure}

Let points $B_1$ and $C_1$ be the points of tangency of lines $AC$ and $AB$ with circle $\omega_a$, respectively. Denote by $X \neq C$ the intersection point of line $AC$ with circle $\Omega$. Moreover, let line $CI$ intersect line $B_1C_1$ at point $J_c$. From Lemma \ref{lemma1} we know that $J_c$ is the excenter of triangle $BCX$ opposite vertex $C$. Furthermore, from the proof of Lemma \ref{lemma1}, we have that quadrilateral $J_cBC_1A'$ is cyclic. Also, note that:

\begin{equation}
\begin{split}
    \arcangle J_cIB &=\arcangle IBC + \arcangle ICB \\
    &= \frac{1}{2} (180^{\circ} - \arcangle BAC)\\
    &= \arcangle B_1C_1A\\
    &= \arcangle J_cC_1B.
\end{split}
\end{equation}
Therefore
\begin{equation}
    \arcangle J_cIB = \arcangle J_cC_1B,
\end{equation} 
and hence quadrilateral $J_cBIC_1$ is cyclic. Combining the above, we get that pentagon $J_cBIC_1A'$ is cyclic, and in particular, quadrilateral $A'C_1IB$ is cyclic. Analogously, we show that $A'B_1IC$ is cyclic.

Let point $N$ be the midpoint of arc $XBC$ (Figure \ref{fig5}). Then from the proof of the Lemma \ref{lemma1}, points $A'$, $B_1$, and $N$ are collinear.

\begin{figure}
    \centering
    \includegraphics{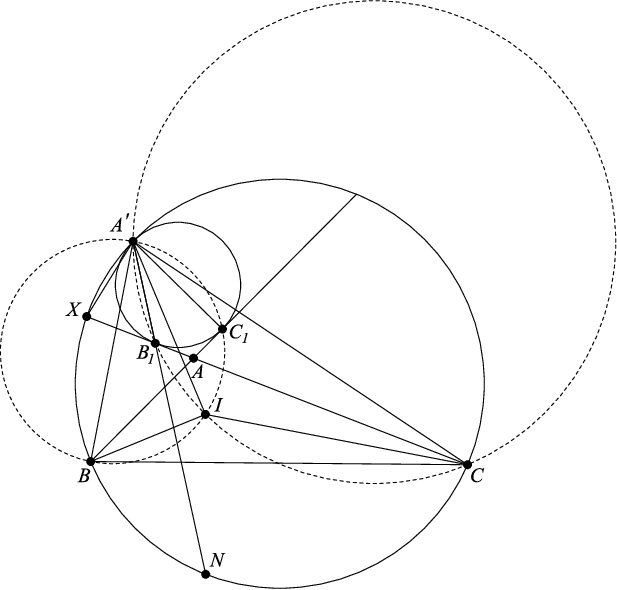}
    \caption{Cyclicity of quadrilaterals $A'C_1IB$ and $AB_1IC$.}
    \label{fig5}
\end{figure}

Therefore, we obtain:
\begin{equation}
    \begin{split}
        \arcangle BA'B_1 &= \frac{1}{2}\arcangle XA'C - \arcangle XA'B \\ &= \frac{1}{2} \arcangle XA'B + \frac{1}{2} \arcangle BA'C - \arcangle XA'B \\
        &= \frac{1}{2}\arcangle BA'C - \frac{1}{2} \arcangle XCB \\ &=\frac{1}{2}\arcangle BA'C - \arcangle B_1CI \\ &= \frac{1}{2}\arcangle BA'C - \arcangle B_1A'I.
    \end{split}
\end{equation}
Thus, 
\begin{equation}
\begin{split}
      \arcangle BA'I &=  \arcangle BA'B_1 + \arcangle B_1A'I  \\ &=  \frac{1}{2}\arcangle BA'C - \arcangle B_1A'I + \arcangle B_1A'I  \\ &= \frac{1}{2}\arcangle BA'C.
\end{split}
\end{equation}
Hence, $A'I$ is the angle bisector of angle $\arcangle BA'C$, as claimed.
\end{proof}

We now present the proof of the generalization of the Sawayama Lemma---the first of the two main results of this paper. 

\begin{thm}[Generalized Sawayama Lemma]\label{saw}
In triangle $ABC$, let $\omega_a$ be a circle tangent to sides $AB$ and $AC$, and let $\Omega$ be a circle passing through $B$ and $C$ and tangent to $\omega_a$. Let $\omega$ be a circle tangent to segment $BC$ at $D$, to $\Omega$, and to a line $k$ tangent to $\omega_a$ at $E$. Then, for any choice of line $k$, the line $DE$ always passes through a fixed point $I$, the incenter of triangle $ABC$.

\end{thm}

\begin{figure}[H]
    \centering
    \includegraphics{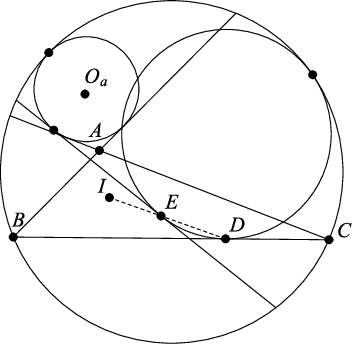}
    \caption{Theorem~\ref{saw}---Configuration 1.}
    \label{fig6}
\end{figure}
This theorem can be viewed as the Sawayama Lemma with vertex $A$ replaced by a circle, which provides a strong generalization.\\

Before proceeding with the proof, we need to specify the configuration of Theorem~\ref{saw}. As will become clear in the proof, this generalization relies heavily on properties of Laguerre geometry. In this geometry, two objects are said to be tangent if and only if they are tangent in the classical Euclidean sense and are similarly directed at the point of tangency. Hence, the configuration can be succinctly described as follows: "Objects (circles or lines) can be oriented so that all tangencies hold in the sense of Laguerre geometry."\\

\noindent Specifically, the configuration can be described as follows:\\

\noindent Circle~$\omega$ lies on the same side of line $BC$ as point $A$. Then:
  \begin{enumerate}
      \item if the circle $\omega_a$ lies on the opposite side of line $AC$ from point $B$, then the circle $\Omega$ is internally tangent to $\omega_a$, and the circle $\omega$ lies on the same side of line $k$ as $\omega_a$,
      \item if the circle $\omega_a$ lies on the same side of line $AC$ as point $B$, then the circle $\Omega$ is externally tangent to $\omega_a$, and the circle $\omega$ lies on the opposite side of line $k$ from $\omega_a$.
  \end{enumerate}

Thus, we have two configurations (Figures~\ref{fig6} and~\ref{fig7}). In the proof, we shall see a natural derivation of the above configuration.

\begin{figure}[H]
    \centering
    \includegraphics{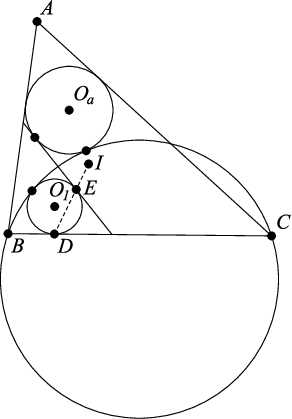}
    \caption{Theorem~\ref{saw}---Configuration 2.}
    \label{fig7}
\end{figure}

\begin{proof}
To simplify the proof, we adopt the configuration shown in Figure \ref{fig6}. The proof for the configuration in Figure \ref{fig7} is analogous.

If the circle $\omega_a$ is a point (i.e., its radius is zero), the problem becomes trivial, as it reduces to the classical Sawayama Lemma. Therefore, we assume from now on that this circle is non-degenerate.

Let points $O$ and $O_a$ denote the centers of circles $\omega$ and $\omega_a$, respectively.

\begin{figure}[H]
    \centering
    \includegraphics[]{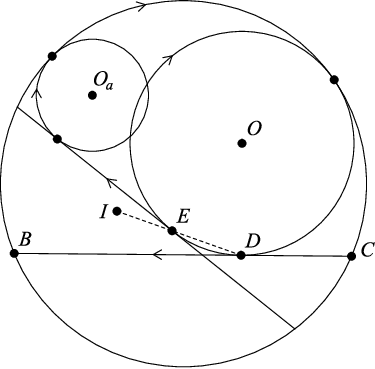}
    \caption{Before transformation.}
    \label{fig8}
\end{figure}

\begin{figure}[H]
    \centering
    \includegraphics[]{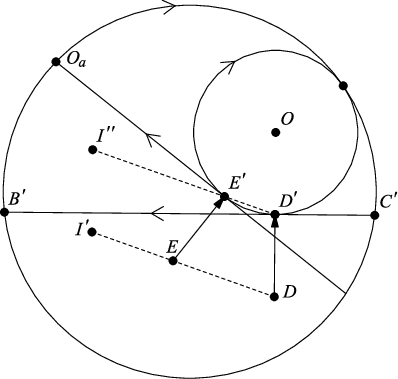}
    \vspace{-14pt}
    \caption{After transformation.}
    \label{fig9}
\end{figure}

Orient the circles and lines as in Figure \ref{fig8}, and then consider a transformation of the plane --- dilation (axial circular transformation)\cite{t4} mapping $\omega_a$ to the point $O_a$.
     
The situation after the transformation is shown in Figure \ref{fig9}. Let $B'$ and $C'$ be the intersection points of the circle $\Omega$ and line $BC$ after the transformation. Let $D'$ and $E'$ be the points of tangency of $\omega$ with lines $BC$ and $k$, respectively, after the transformation. For clarity, it is worth noting that it is not true that $D \rightarrowtail D'$ or $E \rightarrowtail  E'$ under this dilation, since our transformation maps lines to lines, not points to points (in particular, under dilation, a point corresponds to a circle with radius zero). From the Sawayama lemma applied to the triangle $O_aB'C'$, we obtain that line $E'D'$ passes through a fixed point $I''$, which is the incenter of triangle $O_aB'C'$.

Moreover, we mark in Figure \ref{fig9} the original points $D$ and $E$ from before the transformation (from Figure \ref{fig8}). Note that by the definition of dilation, the vector length $DD'$ equals $EE'$ (we shifted the line $BC$ by the same amount as line $k$). Since the points $O,\; E',\; E$ and $O,\; D',\; D$ are collinear, and $OE' = OD'$ and $EE' = D'D$, we conclude that $E'D' \parallel ED$.

Note that regardless of the choice of tangent $k$, the vector $DD'$ is constant --- its direction is perpendicular to $BC$, its orientation is fixed (since circles $\omega$ and $\omega_a$ have fixed orientation), and its length equals the radius of $\omega_a$, which is also constant. Consider now translating the lines $E'D'$ by the fixed vector $\overrightarrow{D'D}$ (i.e., the opposite of the vector $\overrightarrow{DD'}$). Then these lines transform into the lines $DE$, and the fixed point $I''$ transforms into a certain fixed point $I'$. Hence, regardless of the choice of line $k$, the lines $DE$ pass through some fixed point $I'$. It remains to show that $I' = I$.

Now consider the line $k$ tangent to circle $\omega_a$ at point $A'$. Let $M$ be the midpoint of arc $BC$ not containing point $A'$. Then points $A'$, $D$, and $M$ are collinear. Hence, point $I'$ lies on the angle bisector of angle $BA'C$.

\begin{figure}[H]
    \centering
    \includegraphics{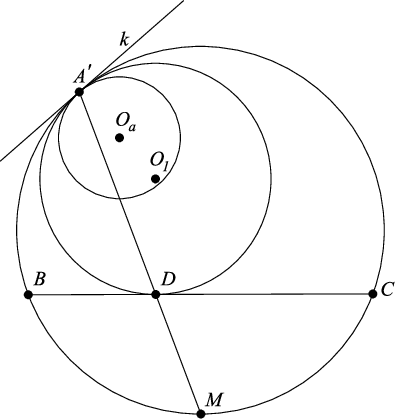}
    \caption{Line $k$ tangent at point $A'$.}
    \label{fig10}
\end{figure}

Let $P$ and $F$ be the point of tangency between circles $\Omega$ and $\omega$, and the point of tangency between circle $\omega_a$ and line $k$, respectively. Moreover let point $A'$ be the point of tangency of circles $\Omega$ and $\omega_a$.

\begin{lem}\label{lemma3}
Points $F$, $I'$, $E$, $P$, and $A'$ lie on the same circle (Figure \ref{fig11}).
\end{lem}
\begin{proof}
\begin{figure}[H]
    \centering
    \includegraphics[scale=0.97]{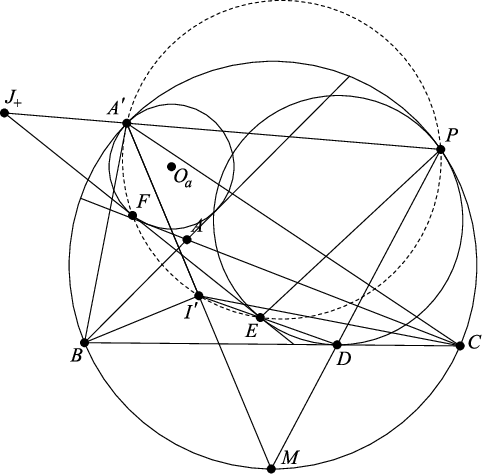}
    \caption{Concyclicity of points $F$, $I'$, $E$, $P$, $A'$.}
    \label{fig11}
\end{figure}

We first show that quadrilateral $PEFA'$ is cyclic. Let $J_+ = k \cap A'P$. By Monge’s Theorem (see \cite{chen}) for circles $\Omega$, $\omega_a$, and $\omega$, the center of the positive homothety mapping $\omega_a$ to $\omega$ lies on line $PA'$. Clearly, this center also lies on $k$. Therefore, $J_+$ is the center of the positive homothety mapping $\omega_a$ to $\omega$, and thus also the center of the positive inversion mapping $\omega$ to $\omega_a$. Under this mapping, $P \rightarrowtail A'$ and $E \rightarrowtail F$. Thus, by the power of a point, quadrilateral $PEFA'$ is cyclic.

Next, we show that quadrilateral $A'PEI'$ is cyclic. Again, we denote by $M$ the midpoint of arc $BC$ of circle $\Omega$ not containing point $A'$. Then line $DP$ passes through point $M$. Furthermore, the points $A'$, $I'$, and $M$ are collinear. Consider the homothety mapping $\omega$ to $\Omega$. We have:
\begin{equation}
    \arcangle DEP = \arcangle MA'P = \arcangle I'A'P,
\end{equation}
so quadrilateral $A'PEI'$ is cyclic.

Since points $P$, $E$, $F$, and $A'$ lie on the same circle and points $A'$, $P$, $E$, and $I'$ lie on the same circle, then points $P$, $E$, $I'$, $F$, $A'$ are concyclic.
\end{proof}

\begin{figure}[H]
    \centering
    \includegraphics{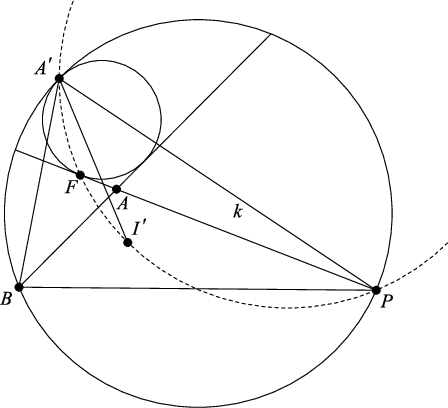}
    \caption{Degenerate circle $\omega = C$.}
    \label{fig12}
\end{figure}

From Lemma \ref{lemma3}, the points $A'$, $F$, $I'$, and $P$ lie on the same circle for any position of point $P$ corresponding to a non-degenerate circle $\omega$. Consider a line $k$ passing through point $C$ such that points $A'$ and $B$ lie on opposite sides of it. Then, circle $\omega$ degenerates to point $C$, so $P = C$. By continuity of cyclicity for points $A'$, $F$, $I'$, and $P$, we get that $A'$, $F$, $I'$, and $P = C$ lie on the same circle (otherwise, we would find a point $P$ on arc $A'C$ not containing $B$ such that circle $A'FI'P$ is not cyclic). From the proof of Lemma \ref{lemma2}, we know that points $A'$, the point of tangency of $AC$ with $\omega$ (here, point $F$), point $I$, and point $C$ lie on the same circle. Thus, both points $I$ and $I'$ lie on the circumcircle of triangle $A'FC$ and on line $A'I$. Since both points are distinct from $A'$, it follows that $I' = I$, which proves that line $DE$ passes through a fixed point --- the incenter of triangle $ABC$.

\end{proof}

\section{Generalization of the Sawayama--Thébault Theorem}
In this section, we present a proof of the Generalization of the Sawayama--Thébault Theorem, which constitutes the second main result of this paper. The key step in the proof is the application of the Generalization of the Sawayama Lemma (Theorem~\ref{saw}), proved in the previous section, similarly to how Yazaburo Sawayama used the classical version (Lemma~\ref{p1}) in his original proof. Moreover, we employ the theory of projective involutions~\cite{coxeter}.

First, we prove a lemma that will be useful in the proof of the Generalization of the Sawayama--Thébault Theorem.

\begin{lem}\label{lem4}
Let circles $\Omega$ and $\omega$ be internally tangent at point $A$, and let $B \neq A$ be a point in the plane. Let $X$ be a point on circle $\omega$, and let $f$ be an involution on that circle. Define $Y = (AXB)\cap \Omega$ (the second intersection point), and $Y' = (Af(X)B)\cap \Omega$ (the second intersection point). Then the map 
\begin{equation}
    Y(\Omega) \rightarrowtail Y'(\Omega)
\end{equation}
is an involution.
\end{lem}

\begin{proof}
Consider an inversion centered at point $A$ with any nonzero radius. This gives us the situation shown in Figure \ref{fig16}. Reflecting point $X$ through point $B$ onto line $\Omega'$, we see that $Y$ and $Y'$ form an involutive pair on line $\Omega'$. Therefore, the map $Y(\Omega) \rightarrowtail Y'(\Omega)$ is also an involution, as desired.
\end{proof}

\begin{figure}[H]
  \centering
  \begin{minipage}{0.45\textwidth}
    \centering
    \includegraphics{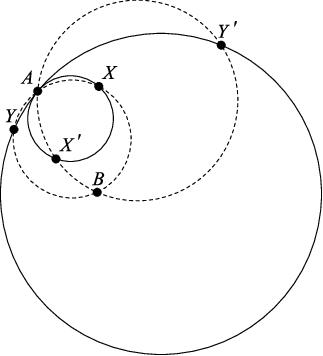}
    \caption{Before inversion.}
    \label{fig15}
  \end{minipage}\hfill
  \begin{minipage}{0.45\textwidth}
    \centering
    \vspace{74pt}
    \includegraphics{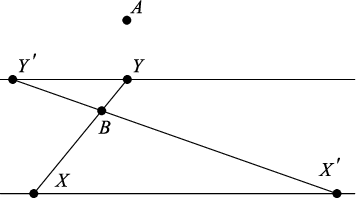}
    \caption{After inversion.}
    \label{fig16}
  \end{minipage}
\end{figure}

Now we present the final result of this paper.

\begin{thm}\label{th32}
\textbf{(Generalized Sawayama--Thébault Theorem)}  
In triangle $ABC$, let $\omega_a$ be a circle tangent to sides $AB$ and $AC$, and let $\Omega$ be a circle passing through $B$ and $C$ and tangent to $\omega_a$.
Let $P$ be any point on a fixed line~$t$. Circles $\omega_1$ and $\omega_2$ are tangent to segment $BC$, to circle $\Omega$, and respectively to lines $k$ and $l$. Then, for any choice of point $P$ on $t$, the line $O_1O_2$ passes through a fixed point $I$, the incenter of triangle $ABC$.
\end{thm}

\begin{figure}[H]
  \centering
    \centering
    \includegraphics{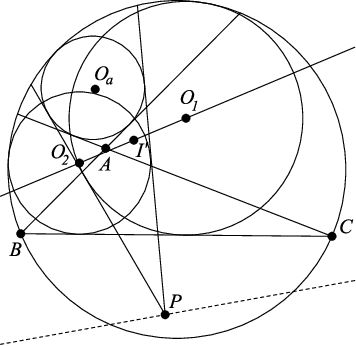}
    \vspace{15pt}
    \caption{Theorem~\ref{th32}---configuration 1.}
    \label{fig13}
\end{figure}

Again, before proceeding with the proof, we need to specify the configuration of Theorem~\ref{th32}. It suffices to note that the configuration is the same as in Theorem~\ref{saw}. Therefore, for clarity, we do not provide a detailed description here. We only assume that the fixed line $t$ does not intersect the circle $\omega_a$, since point $P$ must lie outside this circle in order to draw tangents.

\begin{figure}[H]
    \centering
    \includegraphics[scale=0.82]{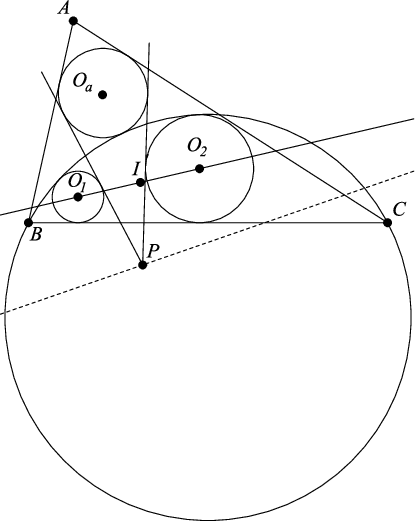}
    \caption{Theorem~\ref{th32}---configuration 2.}
    \label{fig14}
\end{figure}

\begin{proof}
If circle $\omega_a$ is degenerate (i.e., a point $A$), then we obtain the classical Sawayama--Thébault Theorem. From now on, we assume that $\omega_a$ is non-degenerate.

To simplify the proof, we assume the configuration from Figure \ref{fig13}. The proof for the configuration in Figure \ref{fig14} is analogous.

Let us define the following points (Figure \ref{fig17}):  
\begin{itemize}
    \item $D$ --- point of tangency of circle $\omega_1$ with $BC$, 
    \item $E$ --- point of tangency of circle $\omega_1$ with line $k$,
    \item $F$ --- point of tangency of circle $\omega_2$ with $BC$,  
    \item $G$ --- point of tangency of circle $\omega_2$ with line $l$, 
    \item $X$ --- point of tangency of line $k$ with circle $\omega_a$, 
    \item $X'$ --- point of tangency of line $l$ with circle $\omega_a$,
    \item $I$ --- the incenter of triangle $ABC$,  
    \item $A'$ --- point of tangency of circle $\omega_a$ with $\Omega$,
    \item $Y$ --- point of tangency of $\Omega$ and $\omega_1$,
    \item $Y'$ --- point of tangency of $\Omega$ and $\omega_2$,
    \item $O$ --- center of the circumcircle of triangle $A'BC$.
    
\end{itemize}

Now consider the polar transformation with respect to circle $\omega_a$. Since point $P$ lies on a fixed line, the map 
\begin{equation}
    X(\omega_a) \rightarrowtail X'(\omega_a)
\end{equation}
is an involution. From the Lemma \ref{lemma3}, we obtain that the pentagons $A'XEIY$ and $A'X'GIY'$ are cyclic. Moreover, the points $I$ and $A'$ are fixed, independent of the position of point $P$. By Lemma \ref{lem4}, with $\Omega = \Omega$, $\omega = \omega_a$, $A = A'$, $B = I$, and the involution
\begin{equation}
    f := X(\omega_a) \rightarrowtail X'(\omega_a),
\end{equation}
we conclude that the map 
\begin{equation}
    Y(\Omega) \rightarrowtail Y'(\Omega)
\end{equation}
is an involution.

Let line $m$ be the line parallel to $BC$, at distance $R$ from it, and lying on the same side of line $BC$ as point $A'$, where $R$ is the radius of the circle $\Omega$. Consider the parabola $\mathcal{P}$ with directrix $m$ and focus $O$. By construction, $B \in \mathcal{P}$ and $C \in \mathcal{P}$. Let 
\begin{equation}
    D' = O_1D \cap m.
\end{equation} 
Note that:
\begin{equation}
    O_1D' = DD' - O_1D = R - O_1Y = O_1O \implies O_1 \in \mathcal{P}.
\end{equation}
Similarly, we show that $O_2 \in \mathcal{P}$.

\begin{figure}
    \centering
    \includegraphics[scale=0.8]{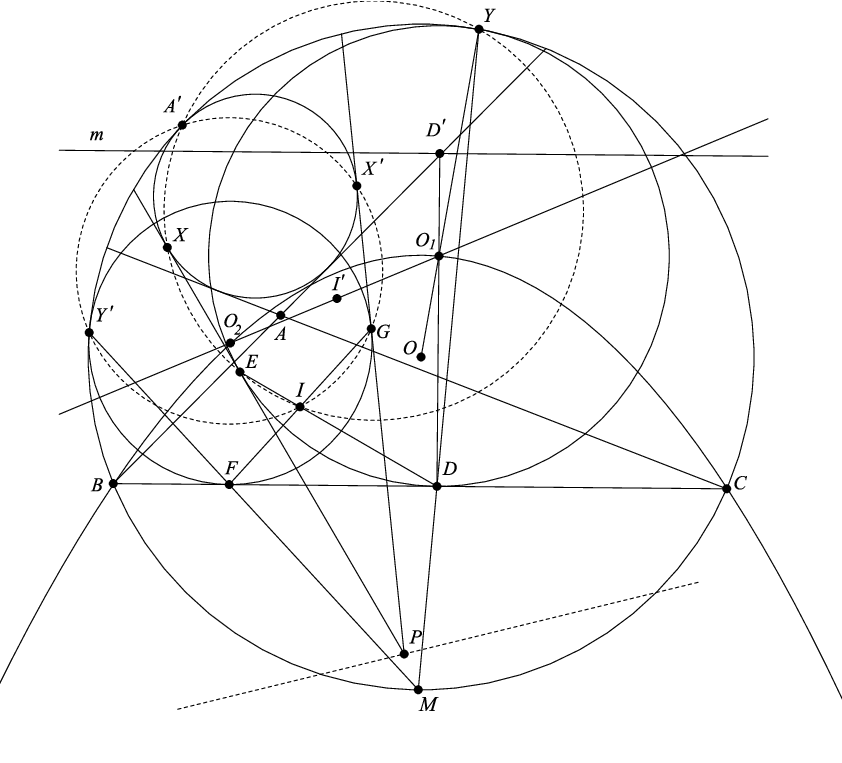}
    \caption{Generalization of the Sawayama--Thébault Theorem.}
    \label{fig17}
\end{figure}

Let $M$ be the midpoint of arc $BC$ of circle $\Omega$ not containing point $A'$, and let $\infty_{DD'}$ be the point at infinity on line $DD'$. From the perpendicularity of $DD'$ to the directrix $m$, we obtain $\infty_{DD'} \in \mathcal{P}$. Now consider the following projective map:

\begin{equation}
    \begin{split}
        Y(\Omega) &\rightarrowtail MY(M) 
        \\&\rightarrowtail MY\cap BC (BC) 
        \\&\rightarrowtail D(BC) 
        \\&\rightarrowtail\infty_{DD'}D(\infty_{DD'}) 
        \\&\rightarrowtail O_1(\mathcal{P}).
    \end{split}
\end{equation}
Since for this projective transformation we have:
\begin{equation}
    Y(\Omega) \rightarrowtail O_1(\mathcal{P}) \quad \text{and} \quad Y'(\Omega) \rightarrowtail O_2(\mathcal{P}),
\end{equation}
and the map 
\begin{equation}
    Y(\Omega) \rightarrowtail Y'(\Omega) 
\end{equation}
is an involution, it follows that the map 
\begin{equation}
    O_1(\mathcal{P}) \rightarrowtail O_2(\mathcal{P})
\end{equation} is also an involution. Hence, by the definition of involution on a conic, the line $O_1O_2$ passes through a fixed point, which completes the proof.

\end{proof}

\begin{rem}
The point $I'$ (the fixed point on line $O_1O_2$), as it depends on the chosen line along which point $P$ moves, is not a distinguished point. The easiest way to determine it is by considering two degenerate cases: $P = AC \cap t$ and $P = AB \cap t$.
\end{rem}

{\noindent\LARGE{\textbf{Statements and Declarations}}}\\

\noindent\textbf{Conflict of interest} The authors declare that they have no conflict of interest.\\

\noindent\textbf{Competing interests} The authors declare no competing interests.



% ------------------------------------------------------------------------
\end{document}